\documentclass[twoside,a4paper,reqno,11pt]{amsart}
\usepackage[top=28mm,right=30mm,bottom=28mm,left=30mm]{geometry}
\usepackage{amsfonts, amsmath, amssymb, mathrsfs, bm, latexsym, stmaryrd, array,  mathtools, bold-extra, xcolor}

\usepackage[colorlinks,linkcolor=cyan,anchorcolor=blue,citecolor=red,backref=page]{hyperref}
\usepackage{caption}
\usepackage{listings}
\usepackage{arydshln}
\usepackage[all]{xy}
\usepackage{appendix}

\usepackage{enumitem}
\setlist[enumerate]{font={\rm},itemsep=0.2\baselineskip}
\setlist[enumerate,1]{label={(\roman*)}}
\setlist[enumerate,2]{label={(\arabic*)}}


\newtheorem{theorem}{Theorem}[section]

\newtheorem{proposition}[theorem]{Proposition}

\newtheorem{lemma}[theorem]{Lemma}

\theoremstyle{definition}

\newtheorem{construction}[theorem]{Construction}

\newtheorem*{conj*}{Conjecture}

\newtheorem{problem}[theorem]{Problem}
\newtheorem*{problem*}{Problem}

\newtheorem{hypothesis}[theorem]{Hypothesis}

\theoremstyle{remark}

\def\GL{{\mathrm{GL}}}
\def\GammaL{{\Gamma \mathrm{L}}}
\def\AGL{{\mathrm{AGL}}}

\def\SL{{\mathrm{SL}}}
\def\PSL{{\mathrm{PSL}}}

\def\Sp{{\mathrm{Sp}}}

\def\PSU{{\mathrm{PSU}}}

\def\G{{\mathrm{G}}}

\def\ASL{{\mathrm{ASL}}}

\def\A{{\mathrm{A}}}

\def\H{{\mathrm{H}}}
\def\M{{\mathrm{M}}}
\def\P{{\mathrm{P}}}

\def\bbZ{{\mathbb{Z}}}
\def\bbF{{\mathbb{F}}}

\def\Sym{{\mathrm{Sym}}}

\def\bfN{{\mathbf{N}}}

\def\Gal{{\mathrm{Gal}}}

\def\leqs{\leqslant}
\def\geqs{\geqslant}

\def\Magma{{\sc Magma}}

\def\normeq{\trianglelefteqslant}

\def\calH{{\mathcal H}}

\title{Finite imprimitive rank $3$ affine groups -- I}

\author{Cai Heng Li}
\address{SUSTech International Center for Mathematics, and Department of Mathematics, Southern University of Science and Technology, Shenzhen, Guangdong, China}
\email{lich@sustech.edu.cn {\text{\rm(Li)}}}

\author{Luyi Liu}
\address{SUSTech International Center for Mathematics, Southern University of Science and Technology, Shenzhen, Guangdong, China}
\email{12031108@mail.sustech.edu.cn {\text{\rm(Liu)}}}

\author{Hanyue Yi}
\address{Department of Mathematics, Southern University of Science and Technology, Shenzhen, Guangdong, China}
\email{12431017@mail.sustech.edu.cn {\text{\rm(Yi)}}}

\author{Yan Zhou Zhu}
\address{Department of Mathematics, Southern University of Science and Technology, Shenzhen, Guangdong, China}
\email{zhuyz@mail.sustech.edu.cn {\text{\rm(Zhu)}}}

\begin{document}

\begin{abstract}
    This is one of a series of papers which aims to classify imprimitive affine groups of rank $3$.
    In this paper, a complete classification is given of such groups of characteristic $p$ whose the point stabilizers are not $p$-local, which shows that such groups are very rare; Specifically, the only examples are the two non-isomorphic groups of the form $2^4{:}\GL_3(2)$ with a unique minimal normal subgroup.
\end{abstract}

\maketitle

\section{Introduction}

The {\it rank} of a transitive permutation group $X\leqs\Sym(\Omega)$ is the number of orbits of $X$ acting on $\Omega\times \Omega$, which is equal to the number of $X_\omega$-orbits on $\Omega$, where $\omega\in\Omega$.
The study of finite rank $3$ groups dates back to the work of Higman~\cite{higman1964Finite} in the 1960s, and has received considerable attention since then, leading to the following long-standing problem.

\begin{problem}\label{prob-1}
	{\rm Classify finite permutation groups of rank $3$.}
\end{problem}

Indeed some important classes of rank $3$ permutation groups have been classified and well-characterized, refer to~\cite{bannai1971Maximal,foulser1969Solvable,kantor1979Permutation,kantor1982Rank,liebeck1987affine,liebeck1986finite} for the primitive case, \cite{devillers2011imprimitive} for the quasiprimitive case, \cite{baykalov2023Rank} for the innately transitive case.
(Recall that a permutation group is called {\it innately transitive} if it has a transitive minimal normal subgroup.)
Our study of Problem~\ref{prob-1} in~\cite{Gross-conj} has led to a solution of Gross' conjecture regarding the so-called $2$-automorphic $2$-groups, with some new families of rank $3$ groups determined.
The significant outcomes achieved have many important applications to various combinatorial objects, including partial linear spaces \cite{bamberg2021Partial,devillers2005classification,devillers2008classification} and combinatorial designs \cite{biliotti2015designs,dempwolff2001Primitive,dempwolff2004Affine,dempwolff2006Symmetric}.

Recently, a reduction theorem for the classification of general imprimitive rank $3$ groups was recently established in~\cite{huang2025finite}.
It divides the unclassified imprimitive rank $3$ groups into three types.
One major type consists of subgroups of the holomorph of a finite group.
Specifically, when the finite group is an elementary abelian $p$-group, such subgroups are called \textit{affine groups}.
This series of papers aims to classify all rank $3$ affine groups.

Now we give a formal definition of affine groups.
Let $V$ be an $n$-dimensional vector space over $\bbF_p$ with prime $p$.
An affine group $X$ is defined as
\[X=V{:}G\leqslant V{:}\GL(V)\cong p^n{:}\GL_n(p),\]
where $V$ acts on itself by translations and $G\leqslant\GL(V)$ acts by linear transformations.
The rank of $X$ is then equal to the number of orbits of $G$ on $V$.
Hence $X$ has rank $3$ if and only if $G$ has exactly $3$ orbits on $V$.
The affine primitive permutation groups of rank $3$ were classified by Foulser~\cite{foulser1969Solvable} in the solvable case and Liebeck~\cite{liebeck1987affine} in general.
Assume that $G$ is reducible, so that $G$ stabilizes a $d$-subspace $W$ of $V$ with $1\leqs d<n$.
Then there exists a subspace $U<V$ such that $V=W\oplus U$, and
\[G\leqs Q{:}(\GL(W)\times\GL(U))=p^{d(n-d)}{:}(\GL_d(p)\times\GL_{n-d}(p)).\]
Observe that, if $O_p(G)=1$, then $G\lesssim \GL(W)\times\GL(U)$; and if $G= \GL(W)\times\GL(U)$, then $G$ has exactly the following $4$ orbits on $V$:
\[\{0\},\ W\setminus\{0\},\ U\setminus\{0\}, \mbox{ and } V\setminus(W\cup U).\]
Naturally, one would ask {\it if there exist subgroups $G<Q{:}(\GL(W)\times\GL(U))$ such that $O_p(G)=1$ and $G$ has exactly $3$ orbits on $V$?}

We say $G$ is \textit{$p$-local} if $O_p(G)\neq 1$.
If $G$ is $p$-local, then the center of $V{:}O_p(G)$ admits a $G$-invariant subspace of $V$.
Then it follows that $G$ is reducible on $V$ if $O_p(G)\neq 1$, see also~\cite[V. 5.17 Satz]{huppert1967Endliche}.
In this paper, we obtain a complete classification of rank $3$ imprimitive affine groups of characteristic $p$ such that the point stabilizer is not $p$-local.
The outcome shows that such groups are very rare, which is slightly surprising.

\begin{theorem}\label{thm:nonplocal}
Let $G\leqslant\GL(V)=\GL_n(p)$ be reducible with $O_p(G)=1$.
Then $G$ has exactly $3$ orbits on $V=\bbF_p^n$ if and only if $V=\bbF_2^4$ and $G\cong\GL_3(2)$ is conjugate to one of the two subgroups of $\GL_4(2)$ defined in Construction~$\ref{exam:two}$.
\end{theorem}

The layout of the paper is as follows.
The two subgroups of $\GL_4(2)$ stated in Theorem~\ref{thm:nonplocal} will be described in detail in Section~\ref{sec:example}.
In Section~\ref{sec:reducible}, we will collect some preliminary results on reducible linear groups which will be used in the proof of Theorem~\ref{thm:nonplocal} in Section~\ref{sec:proof}.

\subsection*{Notation}
    Suppose that $G$ is a group acting on a set $\Omega$ and $\Sigma\subseteq \Omega$.
    Then we write $G_{(\Sigma)}$ and $G_\Sigma$ for the pointwise and setwise stabilizers of $\Sigma$ in $G$, respectively.
    The induced permutation group of $G$ on $\Omega$ is denoted $G^\Omega$.
    We use $G^{(\infty)}$ for the unique perfect group in the derived series of $G$.
    Denote by $O_p(G)$ the largest normal $p$-subgroup of $G$.

    For an $\bbF G$-module $V$, $V^*$ denotes the dual of $V$.
    Moreover, $\M_{n\times m}(\bbF)$ stands for the set of $(n\times m)$-matrices over $\bbF$.
    Denote by $U\otimes W$ the tensor product of $U$ and $W$.
    For $g\in G$ and $v\in V$, the image of $v$ under the action of $g$ is denoted by $v^g$.
    We write $\bbF^n$ for the vector space of column vectors of dimension $n$ over $\bbF$.
    When we say the matrix form of $g$ on $V$, we mean that the matrix $\M(g)$ represents the left multiplication on $\bbF^n\cong V$.

\subsection*{Acknowledgments}
The authors are grateful to the anonymous referees for their valuable comments and suggestions that have helped to improve the paper.

\section{Examples}\label{sec:example}

In this section, we explicitly describe the two imprimitive rank $3$ affine groups arising in Theorem~\ref{thm:nonplocal}.
These two groups both act on $V=\bbF_2^4$, and their stabilizers are both isomorphic to $\GL_3(2)<\GL_4(2)$.
We remark that these two groups were first observed in~\cite[Example 6.3]{huang2025finite}, and their stabilizers $G_1$ and $G_2$ have the following matrix forms:
\[\begin{aligned}
    G_1&=\left\langle\left(\begin{array}{c:ccc}1&1&0&0\\\hdashline 0&1&0&0\\0&0&1&1\\0&0&0&1\end{array}\right),\left(\begin{array}{c:ccc}1&0&0&0\\\hdashline 0&0&0&1\\0&1&0&0\\ 0&0&1&0\end{array}\right)\right\rangle, \\
    G_2&=\left\langle\left(\begin{array}{ccc:c}1&1&0&0\\0&1&0&0\\0&0&1&1\\\hdashline 0&0&0&1\end{array}\right),\left(\begin{array}{ccc:c}0&0&1&0\\1&0&0&0\\0&1&0&0\\\hdashline 0&0&0&1\end{array}\right)\right\rangle.
\end{aligned}\]
Here, we give a construction for these two groups from the exceptional $2$-transitive actions of $\A_7\leqs \A_8\cong \GL_4(2)$ of degree $15$.

\begin{construction}\label{exam:two}
    For the vector space $V=\bbF_2^4$, let $V^\sharp=V\setminus\{0\}$ and let $\calH$ be the set of the $15$ hyperplanes of $V$.
    Then $\GL(V)=\GL_4(2)\cong \A_8$ is $2$-transitive on both $V^\sharp$ and $\calH$.
    Let $H<\GL(V)=\GL_4(2)$ such that $H\cong\A_7$ ($<\A_8$).
    Then $H$ is also $2$-transitive on both $V^\sharp$ and $\calH$.
    We note that $H$ is stabilized by a graph automorphism $\gamma$, see~\cite[Proposition 4.7.3(i)]{bray2013maximal}.
    Let $v\in V^\sharp$.
    Define 
    \[
        \mbox{$G_1=H_v$, the stabilizer of $v\in V^\sharp$ in $H$, and $G_2=G_1^\gamma<H^\gamma$.}
    \]
    
    Note that $G_1$ is the stabilizer of a $1$-dimensional subspace $W=\{0,v\}$ in $H$, and $G_2$ is the stabilizer of a hyperplane $U\in\calH$ in $H$.\qed
\end{construction}

The following lemma uses the notation introduced in Construction~\ref{exam:two}.
\begin{lemma}\label{prop:two}
    The following statements hold.
    \begin{enumerate}
        \item $G_1\cong\GL_3(2)$ has the following $3$ orbits on $V$:
        \[\mbox{$\{0\}$, $\{v\}$ and $V\setminus \{0,v\}$ (lengths $1$, $1$ and $14$, respectively).}\]
        \item $G_2\cong\GL_3(2)$ has the following $3$ orbits on $V$:
        \[\mbox{$\{0\}$, $U\setminus\{0\}$ and $V\setminus U$ (lengths $1$, $7$ and $8$, respectively).}\]
    \end{enumerate}
\end{lemma}
\begin{proof}
    Note that $G_1$ has index $15$ in $\A_7$.
    Then we have $G_1\cong\PSL_2(7)\cong\GL_3(2)$ by $\mathbb{ATLAS}$~\cite{conway1985ATLAS}.
    Recall that $H$ is $2$-transitive on $V^\sharp$, and $G_1$ is the stabilizer of $v$ in $H$.
    Hence $G_1\cong \GL_3(2)$ has three orbits $\{0\}$, $\{v\}$ and $V\setminus W$ as in part~(i).

    Since $G_2={G_1}^\gamma<H^\gamma=H$, we have that $G_2\cong G_1\cong \GL_3(2)$.
    As $H$ is $2$-transitive on $\calH$ and $G_2$ is the stabilizer of $U\in \calH$ in $H$, we have that $G_2$ is transitive on $\calH\setminus \{U\}$.
    We claim that $G_2$ has no fixed vectors in $V^\sharp$.
    Otherwise, if $G_2$ fixes $u\in V^\sharp$, then $\calH_u$, the set of hyperplanes containing the vector $u$, is $G_2$-invariant and $|\calH_u|=7$, which contradicts the transitivity of $G_2$ on $\calH\setminus\{U\}$.
    Hence $G_2$ has no fixed vectors in $V^\sharp$.
    By $\mathbb{ATLAS}$, the maximal subgroups of $\PSL_2(7)\cong\GL_3(2)$ have index at least $7$.
    This implies that each orbit of $G_2$ on $V$ has length either $1$ or no less than $7$.
    Note that $U\setminus\{0\}$ is $G_2$-invariant of size $7$, and hence it is an orbit of $G_2$.
    Note that $V\setminus U$ is also $G_2$-invariant and has size $8$.
    Hence the possible lengths of orbits of $G_2$ on $V\setminus U$ are $1$, $7$ and $8$.
    Since $G_2$ does not fix any non-zero vector, we have that $V\setminus U$ is an orbit of $G_2$.
    Therefore, $G_2\cong\GL_3(2)$ has $3$ orbits: $\{0\}$, $U\setminus\{0\}$ and $V\setminus U$, as in part~(ii).
\end{proof}

The next lemma shows that $G_1$ and $G_2$ are really exceptional.

\begin{lemma}\label{prop:two-2}
Using the notation defined in Construction~$\ref{exam:two}$ and letting $L=\GL(W)\times \GL(U)$, the following statements hold.
	\begin{enumerate}
	\item $L\cong\GL_3(2)$ has the following $4$ orbits:
	\[\mbox{$\{0\}$, $W\setminus\{0\}$, $U\setminus \{0\}$, and $V\setminus (W\cup U)$ (lengths $1$, $1$, $7$ and $7$, respectively)}\]
	\item $\GL(V)\cong \GL_4(2)$ contains exactly $3$ conjugacy classes of subgroups which are isomorphic to $\GL_3(2)$, with representatives $G_1$, $G_2$ and $L$.
	\item Let $X_1=V{:}G_1$, $X_2=V{:}G_2$ and $X_3=V{:}L$.
    Then
    \begin{enumerate}
        \item $X_1$ has a unique minimal normal subgroup $W$;
        \item $X_2$ has a unique minimal normal subgroup $U$; and
        \item $X_3$ has two minimal normal subgroups $W$ and $U$.
    \end{enumerate}
	\end{enumerate}
\end{lemma}

\begin{proof}
    (i).
    As $\dim W=1$ and $\dim U=3$, it follows that $\GL(W)\cong\GL_1(2)=1$ and $\GL(U)\cong\GL_3(2)$, and hence $L\cong\GL_1(2)\times \GL_3(2)\cong\GL_3(2)$ with orbits given in part~(i).

    (ii).
    By Lemma~\ref{prop:two}, $G_1$ does not stabilize any hyperplanes and $G_2$ does not stabilize any $1$-dimensional subspaces.
    Hence $G_1$, $G_2$ and $L$ are not conjugate to each other.
    Let $G<\GL(V)$ be a subgroup isomorphic to $\GL_3(2)$.
    We only need to show that $G$ is conjugate to one of $G_1$, $G_2$ and $L$ in $\GL(V)$.

    Let $\P_1$ and $\P_3$ be the maximal subgroups of $\GL(V)$ stabilizing a subspace of dimension $1$ and $3$, respectively.
    It can be verified on \Magma~\cite{magma} that the irreducible $\bbF_2 G$-modules are of dimension $1$, $3$ and $8$.
    Hence $V$ has an irreducible $\bbF_2 G$-submodule $V_0$ of dimension $1$ or $3$, and thus $G$ is conjugate to a subgroup of $\P_1$ or $\P_3$.
    We remark that
    \begin{enumerate}
        \item[(a)] $G_1$ stabilizes the $1$-dimensional subspace $W$, and so it is conjugate to some subgroup of $\P_1$;
        \item[(b)] $G_2$ stabilizes the $3$-dimensional subspace $U$, and so it is conjugate to some subgroup of $\P_3$;
        \item[(c)] $L$ stabilizes both $W$ and $U$ with $\dim W=1$ and $\dim U=3$, and so it is conjugate to some subgroups of both $\P_1$ and $\P_3$;
        \item[(d)] it can be verified in~\Magma\ that $\P_1\cong \P_3\cong \AGL_3(2)$ has two conjugacy classes of subgroup isomorphic to $\GL_3(2)$.
    \end{enumerate}
    Hence if $G$ is conjugate to a subgroup of $\P_1$, then it is conjugate to either $G_1$ or $L$; and if $G$ is conjugate to a subgroup of $\P_3$, then it is conjugate to either $G_2$ or $L$.
    Thus $G$ is conjugate to one of $G_1$, $G_2$ and $L$.

    (iii).
    Note that the minimal normal subgroup $N$ of an affine group $X=V{:}G$ (with $G\leqs\GL(V)$) is a $G$-invariant subspace of $V$.
    Hence part~(iii) immediately follows from part~(i) and Lemma~\ref{prop:two}.
\end{proof}

\section{Preliminaries}\label{sec:reducible}

Assume that $V=\bbF_p^n$ and $G\leqs \GL(V)$.
It is known that the affine group $V{:}G\leqs\AGL(V)$ is imprimitive if and only if $G$ is reducible, see~\cite[Page 230]{praeger1993ONanScott}.
\begin{proposition}\label{prop:redu}
    Assume that $G$ stabilizes a non-trivial subspace $W<V$.
    Then the following statements hold.
    \begin{enumerate}
        \item $V{:}G$ has rank $3$ if and only if $G$ has the following $3$ orbits on $V$:
        \[\{0\},\ W\setminus \{0\}\mbox{ and }V\setminus W.\]

        \item The set of cosets $V/W=\{v+W: v\in V\}=\{\overline{v}:v\in V\}$ is a block system of $V{:}G$, where $\overline{v}=v+W=\{v+w:w \in W\}$.

        \item If $V{:}G$ has rank $3$, then $G^W$ and $G^{V/W}$ are transitive linear groups.
    \end{enumerate}
\end{proposition}
\begin{proof}
    By definition, the rank of $V{:}G$ is equal to the number of orbits of $G$ acting on $V$.
    Since both $\{0\}$ and $W$ are $G$-invariant, part~(i) immediately follows by counting the number of orbits of $G$.

    Note that $W$ is now a normal subgroup of $V{:}G$.
    Hence the orbits of $W$ acting on $V$ form a block system, where the action of $W$ on $V$ is by translation (any $w\in W$ maps $v\in V$ to $v+w$).
    It is clear that $\overline{v}$ is an orbit of $W$ for each $v\in V$, and thus $V/W$ is a block system, as in part~(ii).

    Assume that $V{:}G$ has rank $3$.
    By part~(i), we have that $G$ is transitive on $W\setminus\{0\}$, and hence $G^W$ is a transitive linear group.
    Since $V/W$ is a block system of $V{:}G$, the induced permutation group $G^{V/W}$ is well-defined.
    Recall that $G$ is transitive on $V\setminus W$.
    It follows that $G$ is transitive on $V/W\setminus\{\overline{0}\}$, that is, $G^{V/W}$ is a transitive linear group, as in part~(iii).
\end{proof}

The above proposition allows us to make the following hypothesis in this section.

\begin{hypothesis}\label{hypo:2trans}
    Let $G\leqs\GL(V)$ with $V=\bbF_p^n$ such that $G$ stabilizes a non-trivial proper subspace $W$ of dimension $d$.
    Moreover, assume that both $G^{W}$ and $G^{V/W}$ are transitive linear groups.
\end{hypothesis}

Now we give a criterion for $G$ having $3$ orbits on $V$.
\begin{lemma}\label{lem:crit}
Under Hypothesis~$\ref{hypo:2trans}$, $G$ has $3$ orbits on $V$ if and only if
    $G_{\overline{v}}$ acts transitively on $\overline{v}=v+W=\{v+w:w\in W\}$ for some $v\in V\setminus W$.
\end{lemma}
\begin{proof}
    Assume that $G$ has $3$ orbits on $V$.
    Then $G$ is transitive on the set $V\setminus W$.
    Note that each subset $\overline{v}$ for $v\in V\setminus W$ forms a block of $G$ acting on $V\setminus W$.
    This yields that $G_{\overline{v}}$ acts transitively on $\overline{v}=v+W$.

    Conversely, we assume that $G_{\overline{v}}$ acts transitively on $\overline{v}=v+W$ for some $v\in V\setminus W$.
    Note that $\{0\}$ and $W\setminus\{0\}$ are two orbits of $G$ acting on $V$ as $G^W$ is a transitive linear group.
    It suffices to show that $V\setminus W$ forms an orbit of $G$.
    Let $u\in V\setminus W$.
    Then there exists some $g_1\in G$ such that $\overline{u}^{g_1}=\overline{v}$ as $G^{V/W}$ is a transitive linear group.
    Hence we have that $u^{g_1}\in \overline{v}$.
    Since $G_{\overline{v}}$ acts transitively on $\overline{v}$, there exists some $g_2$ such that $(u^{g_1})^{g_2}=v$.
    Thus $G$ acts transitively on $V\setminus W$, and so $G$ acting on $V$ has $3$ orbits: $\{0\}$, $W\setminus\{0\}$ and $V\setminus W$.
\end{proof}

Let $\P[W]$ be the maximal parabolic subgroup of $\GL(V)$ stabilizing the subspace $W$, see~\cite[Page 47]{wilson2009finite} for details.
It is known that
\begin{equation}\label{eq:PW}
    \P[W]=Q{:}L\cong p^{d(n-d)}{:}(\GL_{d}(p)\times \GL_{n-d}(p)),\mbox{ where}
\end{equation}
\begin{enumerate}
    \item[(a)] $Q=O_p(\P[W])$ is the intersection of the kernels of $\P[W]$ acting on $W$ and $V/W$;
    \item[(b)] $L$ is a \textit{Levi subgroup} of $\GL(V)$, that is, $L$ stabilizes a direct sum decomposition $V=W\oplus U$ for some complement $U$ of $W$.
\end{enumerate}

Note that $Q$ is an elementary abelian $p$-group, and hence $Q$ can be naturally viewed as an $\bbF_p\P[W]$-module where the action is given by conjugation of $\P[W]$ on $Q$.
Let $v_1,...,v_n$ be a basis of $V$ such that $v_1,...,v_d$ is a basis of $W$.
With respect to this basis, we can identify every element $x$ in $\P[W]$ as a matrix of form $\begin{pmatrix}A&C\\0&B\end{pmatrix}$, where
\begin{enumerate}
    \item[(a)] $A\in\GL_{d}(p)$ is the matrix form of $x$ acting on $W$ with respect to the basis $v_1,...,v_d$;
    \item[(b)] $B\in\GL_{n-d}(p)$ is the matrix form of $x$ acting on $V/W$ with respect to the basis $v_{d+1},...,v_n$; and
    \item[(c)] $C\in\M_{d\times(n-d)}(\bbF_p)$ is the set of $d\times(n-d)$ matrices over $\bbF_p$.
\end{enumerate}
In particular, we have that each element in $Q$ has matrix form $\begin{pmatrix}I&C\\0&I\end{pmatrix}$, and each matrix in $\M_{d\times(n-d)}(\bbF_p)$ is a linear combination of $e_i\cdot f_j^\top$'s where $e_1,...,e_d$ and $f_1,...,f_{n-d}$ are bases of $\M_{d\times 1}(\bbF_p)$ and $\M_{(n-d)\times 1}(\bbF_p)$, respectively.
We remark that $\P[W]$ naturally acts on vector spaces $W$ and $V/W$, and these actions naturally admit $\bbF_p\P[W]$-module structures on $W$ and $V/W$.
Note that
\[\begin{pmatrix}A&C\\0&B\end{pmatrix}\begin{pmatrix}I&e_if_j^\top\\0&I\end{pmatrix}\begin{pmatrix}A&C\\0&B\end{pmatrix}^{-1}=\begin{pmatrix}I&(Ae_i)((B^{-1})^\top f_j)^\top\\0&I\end{pmatrix}.\]
This yields that $Q\cong W\otimes(V/W)^*$ as $\bbF_p\P[W]$-modules, where $(V/W)^*$ is the dual of $V/W$.
We give the following lemma for $G\leqs\P[W]$.

\begin{lemma}\label{lem:plocal}
    Under Hypothesis~$\ref{hypo:2trans}$, for $Q=O_p(\P[W])$, the following statements hold.
    \begin{enumerate}
        \item $G\cap Q=G_{(W)}\cap G_{(V/W)}=O_p(G)$.
        \item $Q$ is naturally isomorphic to the $\bbF_pG$-module $W\otimes (V/W)^*$, and $O_p(G)$ is an $\bbF_pG$-submodule of $Q$.
    \end{enumerate}
\end{lemma}
\begin{proof}
    Since $Q$ is the intersection of kernels of $\P[W]$ acting on $W$ and $V/W$, we have that $G\cap Q=G_{(W)}\cap G_{(V/W)}$ as $G\leqs \P[W]$.
    Recall that both $G^W$ and $G^{V/W}$ are irreducible.
    This yields that $O_p(G^W)=O_p(G^{V/W})=1$, and then
    \[O_p(G)\leqs G_{(W)}\cap G_{(V/W)}=G\cap Q.\]
    As $Q\normeq \P[W]$, we have that $G\cap Q\normeq G$, and thus $O_p(G)=G\cap Q$, as in part~(i).

    The $\bbF_pG$-module structure of $Q$ is restricted from the $\bbF_p\P[W]$-module structure, and hence part~(ii) clearly holds.
\end{proof}

Denote by $\H^i(A,W)$ the \textit{$i$-th cohomology group} of an $\bbF A$-module $W$.
Remark that $\H^i(A,W)$ is an additive group, $|\H^0(A,W)|$ is equal to the number of fixed-points of $A$ acting on $W$, and $|\H^1(A,W)|$ is equal to the number of conjugacy classes of complements of $W$ in $W{:}A$.
We refer the readers to~\cite{aschbacher1984applications} for more details about properties and applications of first cohomology groups.
Some important properties used in this paper are listed below.
\begin{proposition}\label{prop:coho}
    Let $q=p^f$ with prime $p$, and let $W$ and $U$ be an $\bbF_q A$-module and an $\bbF_q B$-module, respectively.
    Then the following statements hold.
    \begin{enumerate}
        \item If $p\nmid|A|$, then $\H^1(A,W)=0$.
        \item If $N\normeq A$ and $\H^0(N,W)=0$, then $|\H^1(A,W)|\leqs|\H^1(N,W)|$.
        \item (K\"{u}nneth formula~\cite[Proposition 2.5]{cline1975Cohomology}) If $\H^0(A,W)=0$, then
        \[\H^1(A\times B,W\otimes U)\cong \H^1(A,W)\otimes \H^0(B,U).\]
    \end{enumerate}
\end{proposition}

Recall the notation $\P[W]=Q{:}L$ with $L=\GL(W)\times \GL(U)$, see Equation~\ref{eq:PW}.
Similarly to Lemma~\ref{prop:two-2}\,(i), $L$ has $4$ orbits on $V$.
Hence each subgroup of $L$ has at least $4$ orbits on $V$.
Note that $L$ is a complement of $Q$ in $\P[W]$.
If $G<\P[W]$ is not $p$-local, then $G\cap Q=1$ and $G$ is a complement of $Q$ in $Q{:}G<\P[W]$.
These facts imply the following lemma.
\begin{lemma}\label{lem:conjLevi}
    If $O_p(G)=1$ and $\H^1(G,Q)=0$, then $G$ has at least $4$ orbits on $V$.
\end{lemma}
\begin{proof}
    Let $Y=\langle Q,G\rangle$.
    Then $Y=Q{:}G$ as $Q\cap G=O_p(G)=1$ by Lemma~\ref{lem:plocal}\,(i).
    Recall that $\P[W]=Q{:}L$, where $L$ stabilizes the direct sum decomposition $V=W\oplus U$.
    Since $Y\leqs\P[W]$, there exists $L_0\leqs L$ such that $Y=Q{:}L_0$.
    Then $G$ is conjugate to $L_0$ in $Y$ as $\H^1(G,Q)=0$, and hence $G$ and $L_0$ have the same number of orbits on $V$.
    By Lemma~\ref{prop:two-2}\,(i), $L$ has $4$ orbits on $V$.
    So each of $G$ and $L_0$ has at least $4$ orbits on $V$.
\end{proof}

The classification of transitive linear groups was completed in~\cite{hering1985Transitive,huppertZweifachTransitiveAufl1957}.
Here we record the classification of non-solvable transitive linear groups.

\begin{theorem}\label{thm:translinear}
    Let $G\leqs \GL(V)$ be a non-solvable transitive linear group acting on $V\cong\bbF_p^n$.
    Then one of the following holds.
    \begin{enumerate}
        \item $G$ belongs to one of the three infinity families:
        \begin{enumerate}
            \item $n=mf$ and $G^{(\infty)}\cong \SL_m(p^f)$ for $m\geqs 2$ and $(m,p^f)\notin\{(2,2),(2,3)\}$;
            \item $n=2mf$ and $G^{(\infty)}\cong \Sp_{2m}(p^f)$ for $2m\geqs 4$ and $(2m,p^f)\neq (4,2)$;
            \item $n=6f$ and $G^{(\infty)}\cong \G_2(2^f)$ for $f>1$.
        \end{enumerate}
        \item $G^{(\infty)}$ and $(n,p)$ are listed in Table~$\ref{tab:sporadic}$. 
    \end{enumerate}
\end{theorem}

\begin{table}[ht]
    \centering
    \caption{Non-solvable transitive linear groups in Theorem~\ref{thm:translinear}\,(ii)}\label{tab:sporadic}
    \begin{tabular}{cccc}
    \hline
        $(n,p)$  & $G^{(\infty)}$ & number of possible $G$ \\
        \hline
        $(2,11)$ & $\SL_2(5)$ & $2$\\
        $(2,19)$ & $\SL_2(5)$ & $1$\\
        $(2,29)$ & $\SL_2(5)$ & $2$\\
        $(2,59)$ & $\SL_2(5)$ & $1$\\
        $(4,2)$  & $\A_6\cong \Sp_4(2)'$ & $2$ \\
        $(4,2)$  & $\A_7$ & $1$ \\
        $(4,3)$  & $\SL_2(5)$ & $4$\\
        $(4,3)$  & $2_{-}^{1+4}.\A_5$ & $2$ \\
        $(6,2)$  & $\PSU_3(3)\cong \G_2(2)'$ & $2$\\
        $(6,3)$  & $\SL_2(13)$ & $1$\\
        \hline
    \end{tabular}
\end{table}

The following properties of non-solvable transitive linear groups can be observed from the classification.

\begin{proposition}\label{prop:transList}
    Let $X\leqs\GL(U)$ and $Y\leqs\GL(V)$ be non-solvable transitive linear groups over $\bbF_p$.
    Assume that $u\in U\setminus\{0\}$ and $v\in V\setminus\{0\}$.
    If $X$ and $Y$ have isomorphic non-abelian simple composition factors, then
    \[\dim U=\dim V,\ X^{(\infty)}\cong Y^{(\infty)},\mbox{ and }(X^{(\infty)})_u\cong(Y^{(\infty)})_v.\]
\end{proposition}

Recall from Lemma~\ref{lem:crit} that if $G$ has $3$ orbits on $V$, then $G_{\overline{v}}$ acts transitively on $\overline{v}=v+W$, and hence $G_{\overline{v}}$ has a subgroup of index $|\overline{v}|=p^d$.
The following lemma is from~\cite[Corollary 2.12 \& Lemma 2.14]{huang2025finite}.
It can be  deduced from Guralnick's classification of finite simple groups with a prime power index subgroup~\cite{guralnick1983Subgroups}.
We remark that the groups $\Sp_4(2)'\cong \A_6$ and $\G_2(2)'\cong \PSU_3(3)$ from \cite[Corollary 2.12]{huang2025finite} are not listed here because $2^4\nmid |\A_6|$ and $2^6\nmid |\PSU_3(3)|$.

\begin{lemma}\label{lem:transindexpd}
    Suppose $T\leqs\GL_n(p)$ is a transitive linear group.
    If $p^d$ divides $|T|$ with $d\geqslant 1$, then $T$ is non-solvable and one of the following holds.
    \begin{enumerate}
        \item $\SL_m(p^f)\normeq T\leqs \GammaL_m(p^f)$ for $m\geqs 3$ and $n = mf$.
        \item $\Sp_{m}(p^f)\normeq T$ for $m\geqs 4$ and $n = mf$.
        \item $\G_2(2^f)\normeq T$ and $(n,p) = (6f,2)$.
    \end{enumerate}

    Moreover, if $T$ has a subgroup of index $p^n$, then $T= \GL_3(2)$ with $(n,p)=(3,2)$.
\end{lemma}

In the rest of this section, we prove the following lemma which will be used in the subsequent proof.

\begin{lemma}\label{lem:transtabindex}
    Suppose $T\leqs\GL_n(p)$ is a transitive linear group.
    If $T_v$ has a subgroup of index $p^n$ for some non-zero vector $v\in V$, then $T$ is non-solvable and one of the following statements holds.
    \begin{enumerate}
        \item $\SL_n(p)\normeq T$ with $(n,p)\in \{(3,2),(3,3),(3,5),(3,7),(3,11)\}$.
        \item $\Sp_{n}(p)\normeq T$ with $(n,p)\in \{(4,2),(4,3),(4,5),(4,7),(4,11),(6,2)\}$.
        \item $T=\G_2(2)$ with $(n,p)=(6,2)$.
    \end{enumerate}
\end{lemma}
\begin{proof}
    Lemma~\ref{lem:transindexpd} implies that $T$ is non-solvable and one of the following holds.
    \begin{enumerate}
        \item[(a)] $H=\SL_{n/f}(p^f)$ is normal in $T$ such that $n/f\geqs 3$.
        \item[(b)] $H=\Sp_{n/f}(p^f)$ is normal in $T$ with $n/f\geqs 4$.
        \item[(c)] $H=\G_2(2^f)$ is normal in $T$ with $(d,p)=(6f,2)$.
    \end{enumerate}
    We remark that $H$ is a transitive linear group and $T/H\lesssim\GammaL_1(p^f)$ in these three cases.
    Hence we have that
    \[T_v/H_v=T_v/(H\cap T_v)\cong HT_v/H\leqslant T/H\lesssim\GammaL_1(p^f)\]
    Then it follows that $|T_v/H_v|$ divides $(p^f-1)f$, and hence it is not divisible by $p^n$ as $f<p^f<p^n$.
    Recall that $T_v$ has a subgroup of index $p^n$.
    It follows that $H_v$ has a subgroup $L$ of index $p^k$, where $k\leqs n$ and $p^{n-k}\leqs f$.
    We will verify this condition case by case, and we remind the readers that the structures of $H_v$ can be found in~\cite{wilson2009finite}.

    Assume that case~(a) holds.
    We have that
    \[H_v\cong \ASL_{d}(p^f)\cong p^{df}{:}\SL_{d}(p^f)\mbox{ with $d=\frac{n-f}{f}=\frac{n}{f}-1$}.\]
    Note that $|O_p(H_v)|=p^{df}=p^{n-f}$ and $n-f<k$ as $p^{n-k}\leqs f<p^f$.
    This implies that $LO_p(H_v)/O_p(H_v)\cong L/(L\cap O_p(H_v))$ has index $p^t$ in $H_v/O_p(H_v)\cong \SL_d(p^f)$ for some $t\geqslant 1$.
    By Guralnick's classification~\cite{guralnick1983Subgroups}, we conclude that $f=1$ (hence $H=\SL_n(p)$) and
    \[(n,p)\in \{(3,2),(3,3),(3,5),(3,7),(3,11),(4,2),(5,2)\}.\]
    We verify the existence of subgroups of index $p^n$ in $T_v$ for each case by computations in \Magma, and obtain that
    \[(n,p)\in \{(3,2),(3,3),(3,5),(3,7),(3,11)\}.\]

    Assume that case~(b) holds.
    Then
    \[H_v\cong  [p^{f+df}]{:}(\bbZ_{p^f-1}\times \Sp_{d}(p^f))\mbox{ with $d=\frac{n-2f}{f}=\frac{n}{f}-2$}.\]
    Note that $O_p(H_v)$ has order $p^{f+df}=p^{n-f}<p^k$.
    By arguments similar to case~(a), we have that $\Sp_d(p^f)$ has a subgroup of index $p^t$ for some positive integer $t$.
    Again by Guralnick's classification and calculations in~\Magma, we conclude that
    \[(n,p)\in \{(4,2),(4,3),(4,5),(4,7),(4,11),(6,2)\}.\]

    Finally, we assume that case~(c) holds.
    Then $H\cong \G_2(2^f)$, and $H_v\cong [2^{5f}]{:}\SL_2(2^f)$ has a subgroup of index $2^k>2^{n-f}=2^{5f}$.
    Similar reasoning yields that $\SL_2(2^f)$ has a subgroup of index $2^t$ for some $t\geqs 1$.
    Then it follows that $\SL_2(2^f)$ is solvable by Guralnick's classification, and hence $f=1$.
    Therefore, $T=H=\G_2(2)$ as in part~(iii).
\end{proof}

\section{Proof of Theorem~\ref{thm:nonplocal}}\label{sec:proof}

In this section, we always assume that $G<\GL(V)$ is reducible such that $V\cong \bbF_p^n$, $O_p(G)=1$ and $G$ has $3$ orbits on $V$.
Let $W\cong \bbF_p^d$ be a $G$-invariant subspace of $V$ with $0<d<n$.
By Proposition~\ref{prop:redu}\,(iii), we have that both $G^W$ and $G^{V/W}$ are transitive linear groups.
Then $G<\P[W]=Q{:}L$ with $Q=O_p(\P[W])$, and further $Q\cong W\otimes (V/W)^*$ as $\bbF_p\P[W]$-modules, see Equation~\ref{eq:PW}.

The following lemma shows that $G$ is faithful on at least one of $W$ and $V/W$.
\begin{lemma}\label{lem:kernel}
    With the above assumptions, one of the following statements holds.
    \begin{enumerate}
        \item[(A)] $G\cong G^W\cong G^{V/W}$, and hence $G_{(W)}=G_{(V/W)}=1$.
        \item[(B)] $G\cong G^{V/W}$ and $G_{(W)}\neq 1$ is almost simple.
        \item[(C)] $G\cong G^W$ and $G_{(V/W)}\neq 1$ is almost simple.
    \end{enumerate}
\end{lemma}
\begin{proof}
    First, we prove that one of $G_{(W)}$ and $G_{(V/W)}$ is trivial.
    Suppose that $G_{(W)}\neq 1$ and $G_{(V/W)}\neq 1$.
    Since $G_{(W)}\cap G_{(V/W)}=O_p(G)=1$ by Lemma~\ref{lem:plocal}\,(i), we have that $K:=\langle G_{(W)},G_{(V/W)} \rangle$ is a direct product $G_{(W)}\times G_{(V/W)}$.
    Note that
    \[(G_{(W)})^{V/W}\cong G_{(W)}/(G_{(W)}\cap G_{(V/W)})\cong G_{(W)}.\]
    This shows that $G_{(W)}$ acts faithfully on $V/W$.
    Similarly, we have that $G_{(V/W)}$ acts faithfully on $W$.
    Now we show that $|\H^0(K,Q)|=1$.
    Recall that
    \[|\H^0(K,Q)|=|\H^0(G_{(W)}\times G_{(V/W)},W\otimes(V/W)^*)|.\]
    Since $G^{W}$ is irreducible on $W$ and $1\neq (G_{(V/W)})^{W}\unlhd G^{W}$, we have that $G_{(V/W)}$ has no fixed points on $W\setminus\{0\}$.
    Then the zero vector of $W\otimes (V/W)^*$ is the unique fixed point of $G_{(V/W)}$.
    This yields that
    \[\H^0(K,Q)\leqs|\H^0(G_{(V/W)},W\otimes(V/W)^*)|=1.\]
    By Proposition~\ref{prop:coho}\,(ii) and (iii), we have that
    \[\begin{aligned}
        |\H^1(G,Q)|&\leqs|\H^1(K,Q)|=|\H^1(G_{(W)},(V/W)^*)\otimes \H^0(G_{(V/W)},W)|\\
        &=|\H^1(G_{(W)},(V/W)^*)\otimes 0|=1.
    \end{aligned}\]
    This implies that $G$ has at least $4$ orbits on $V$ by Lemma~\ref{lem:conjLevi}, a contradiction.

    Since at least one of $G_{(W)}$ and $G_{(V/W)}$ is trivial, we only need to show that $G_{(W)}$ (or $G_{(V/W)}$) is almost simple if it is non-trivial.
    The proofs for these two cases are similar, and we only verify the case where $G_{(W)}\neq 1$.

    Assume that $G_{(W)}\neq 1$.
    Then $G$ acts faithfully on $V/W$, and hence $G\cong G^{V/W}$ has a unique non-abelian simple composition factor as deduced from Theorem~\ref{thm:translinear}.
    It follows that $G_{(W)}$ has at most one non-abelian simple minimal normal subgroup.
    Remark that if $G_{(W)}$ has no solvable normal subgroups, then $G_{(W)}$ has a unique minimal normal subgroup which is non-abelian simple, and hence $G_{(W)}$ is almost simple.
    Thus it suffices to show that $G_{(W)}$ has no solvable normal subgroups.

    Suppose that $G_{(W)}$ has a solvable normal subgroup $S$.
    Then there exists a prime $q$ such that $O_q(S)\neq 1$.
    This yields that $O_q(G_{(W)})\neq 1$ as $O_q(S)\lhd G_{(W)}$.
    Let $T=O_q(G_{(W)})$.
    Note that $T\lhd G$ and $G\cong G^{V/W}$ is a transitive linear group on $V/W$.
    Then the zero vector of $V/W$ is the unique fixed point of $T$.
    Since $T\unlhd G_{(W)}$, Proposition~\ref{prop:coho} implies that
    \[\begin{aligned}
        |\H^1(G,Q)|&=|\H^1(G,W\otimes(V/W)^*)|\leqs|\H^1(T,W\otimes(V/W)^*)|=|\H^1(1\times T,W\otimes(V/W)^*)|\\
        &=|\H^1(T,(V/W)^*)\otimes \H^0(1,W)|=|0\otimes W|=1.
    \end{aligned}\]
    This shows that $G$ has at least $4$ orbits on $V$ by Lemma~\ref{lem:conjLevi}, a contradiction.
    The proof is complete.
\end{proof}

First, we show that there is no group $G$ satisfying Case~(A) of Lemma~\ref{lem:kernel}.

\begin{lemma}\label{lem:bothfaith}
    Exactly one of $G_{(W)}$ and $G_{(V/W)}$ is trivial.
\end{lemma}
\begin{proof}
    By Lemma~\ref{lem:crit}, $G_{\overline{v}}$ acts transitively on $\overline{v}=v+W$ for some $v\in V\setminus W$.
    Since $|\overline{v}|=|W|=p^d$, we have that $G\cong G^W$ has a subgroup of index $p^d$.
    Lemma~\ref{lem:transindexpd} yields that $G$ is non-solvable.
    Since $G^W\cong G^{V/W}$, they have the same composition factors.
    By Proposition~\ref{prop:transList}, $\dim V=n=2d$ with $\dim W=\dim V/W=d$ and $G_{\overline{v}}\cong G_{w}$ for any non-zero vector $w\in W$.
    As $G_{\overline{v}}$ is transitive on $\overline{v}=v+W$, both $G_{\overline{v}}$ and $G_w$ have a subgroup of index $|\overline{v}|=p^d$.
    By Lemma~\ref{lem:transtabindex}, $G$ has a normal subgroup $H$ such that one of the following cases holds.
    \begin{enumerate}
        \item[(a)] $H\cong \SL_d(p)$ with $(d,p)\in \{(3,2),(3,3),(3,5),(3,7),(3,11)\}$.
        \item[(b)] $H\cong \Sp_{d}(p)$ with $(d,p)\in \{(4,2),(4,3),(4,5),(4,7),(4,11),(6,2)\}$.
        \item[(c)] $G=H\cong \G_2(2)$ with $(d,p)=(6,2)$.
    \end{enumerate}
    In these three cases, we observe that $H$ is transitive on both $W\setminus\{0\}$ and $(V/W)\setminus\{\overline 0\}$ and $G/H$ is a $p'$-group.
    This yields that $H_{\overline{v}}\unlhd G_{\overline{v}}$ is transitive on $\overline{v}$.
    Hence $H\leqs G$ has $3$ orbits on $V$ with $O_p(H)=1$ by Lemma~\ref{lem:crit}.
    Lemma~\ref{lem:conjLevi} implies that $\H^1(H,W\otimes (V/W)^*)\cong \H^1(H,Q)\neq 0$.
    By~\cite[Theorem 1.2.2]{group2013First}, we obtain that $p=2$ or $3$.
    Thus we conclude that $H\cong H^W\cong H^{V/W}$ acts on $V$ with $3$ orbits and is isomorphic to one of
    \[\SL_3(2),\ \SL_3(3),\ \Sp_4(2),\ \Sp_4(3),\ \Sp_6(2)\mbox{ and }\G_2(2).\]
    Now, we use \Magma\ to compute such groups $H$ with the following method:
    \begin{enumerate}
        \item[step 1:] for each $H$ above, compute the set $\mathcal{V}$ consisting of irreducible $\bbF_pH$-modules of dimension $d$ on which $H$ is a transitive linear group;
        \item[step 2:] for $(W,U)\in\mathcal{V}^2$,  construct the direct sum $W\oplus U$, and then obtain the corresponding matrix group:
        \[M_1=\left\{\begin{pmatrix}A&0\\0&A^\sigma \end{pmatrix}:A\in H^W\right\},\mbox{ where $\sigma:H^{W}\rightarrow H^{V/W}$ is an isomorphism};\]
        \item[step 3:] construct $M_2=\langle M_1,Q\rangle$ as follows:
        \[M_2=\left\{\begin{pmatrix}A&C\\0&A^\sigma \end{pmatrix}:A\in H^W\mbox{ and }C\in\mathrm{M}_{d\times d}(\bbF_p)\right\};\]
        \item[step 4:] obtain all subgroups isomorphic to $H$ in $M_2$ and calculate their orbits.
    \end{enumerate}
    We conclude the result as follows.
    \begin{enumerate}
        \item[(1)] If $H\cong \SL_3(2)$, then $|\mathcal{V}|=2$.
        When $W\cong V/W$, there is a unique conjugacy class of subgroups isomorphic to $H$ in $M_2$;
        when $W\ncong V/W$, there are two conjugacy classes of subgroups isomorphic to $H$ in $M_2$.
        All such subgroups have $5$ orbits on $V$.
        \item[(2)] If $H\cong \SL_3(2)$, then $|\mathcal{V}|=2$.
        Whenever $W\cong V/W$ or $W\ncong V/W$, there is a unique conjugacy class of subgroups isomorphic to $H$ in $M_2$, and all such subgroups have $6$ orbits on $V$.
        \item[(3)] If $H\cong \Sp_4(2)$, then $|\mathcal{V}|=2$.
        When $W\cong V/W$, there are two conjugacy classes subgroups isomorphic to $H$ in $M_2$, and all such subgroups have $5$ and $6$ orbits on $V$, respectively.
        When $W\ncong V/W$, there is a unique conjugacy class of subgroups isomorphic to $H$ in $M_2$, and all such subgroups have $5$ orbits on $V$.
        \item[(4)] If $H\cong \Sp_4(3)$, then $|\mathcal{V}|=1$ and there are $3$ conjugacy class of subgroups isomorphic to $H$ in $M_2$.
        They have $6$, $6$ and $8$ orbits on $V$, respectively.
        \item[(5)] If $H\cong \Sp_6(2)$, then $|\mathcal{V}|=1$ and there is a unique conjugacy class of subgroups isomorphic to $H$ in $M_2$.
        Every such subgroup has $6$ orbits on $V$.
        \item[(6)] If $H\cong \G_2(2)$, then $|\mathcal{V}|=1$ and there are $2$ conjugacy class of subgroups isomorphic to $H$ in $M_2$.
        All subgroups have $7$ orbits on $V$.
    \end{enumerate}
    Therefore, we complete the proof by exhausting all possibilities.
\end{proof}

Recall that the group $G_2\cong\GL_3(2)$ defined in Construction~\ref{exam:two} acts on $V=\bbF_2^4$ and stabilizes a hyperplane $W$.
Then $V/W$ has dimension $1$, and hence $G_2$ acts trivially on $V/W$.
It follows that $G_2$ satisfies Case~(B) in Lemma~\ref{lem:kernel}.
The following lemma shows that it is the unique possibility for Case~(B).

\begin{lemma}\label{lem:faithfulW}
    Assume that $G_{(W)}=1$ and $G_{(V/W)}\neq 1$ is almost simple.
    Then $V=\bbF_2^4$ and $G$ is conjugate to $G_2$ defined in Construction~$\ref{exam:two}$.
\end{lemma}

\begin{proof}
    Note that $G_{\overline{v}}$ is transitive on $\overline{v}$ of degree $p^d=|\overline{v}|$ by Lemma~\ref{lem:crit}.
    It follows that the order of $G^W\cong G$ is divisible by $p^d$.
    Since $G^W$ is a transitive linear group, we have that $G$ is non-solvable and $G^{(\infty)}$ is a transitive linear group on $W$ by Lemma~\ref{lem:transindexpd}.
    Recall that $G\cong G^W$ has a unique non-abelian simple composition factor deduced by Theorem~\ref{thm:translinear}.
    It follows that $G^{V/W}\cong G/G_{(V/W)}$ is solvable.
    Hence $G^{(\infty)}\leqs G_{(V/W)}$, that is, $G^{(\infty)}$ acts trivially on $V/W$.
    It follows that $G^{(\infty)}$ fixes $\overline{v}$.
    Recall that $G^{(\infty)}$ is transitive on $W\setminus\{0\}$.
    So $G_{\overline{v}}\geqslant G^{(\infty)}$ is also transitive on $W\setminus\{0\}$.
    As $G_{\overline{v}}$ has a subgroup of index $p^d$, we have that $W=\bbF_2^3$ and $G_{\overline{v}}\cong \GL_3(2)$ by Lemma~\ref{lem:transindexpd}.
    Since $G\cong G^W\leqslant\GL_3(2)$, we have that $G=G_{\overline{v}}=\GL_3(2)$.
    Recall that $G_{(V/W)}\normeq G$ is almost simple, it follows that $G_{(V/W)}=G$, and then $V=\bbF_2^4$ and $\dim V/W=1$.
    Hence $G\cong \GL_3(2)$ is a subgroup of $\GL(V)\cong \GL_4(2)$ stabilizing the hyperplane $W$ and has $3$ orbits on $V$.
    Lemma~\ref{prop:two-2} yields that $G$ is conjugate to $G_2$ defined in Construction~\ref{exam:two}.
\end{proof}

From now, we focus on Case~(C) in Lemma~\ref{lem:kernel} where $G$ acts faithfully on $V/W$.
First, we prove that $G\cong G^{V/W}$ is a member of the three infinity families of non-solvable groups given in Theorem~\ref{thm:translinear}.

\begin{lemma}\label{lem:reduceclass}
    Assume that $G_{(V/W)}=1$ and $G_{(W)}\neq 1$ is almost simple.
    Then $G\cong G^{V/W}$ is a group satisfying Theorem~{\rm\ref{thm:translinear}\,(i)}
\end{lemma}
\begin{proof}
    Since $G\cong G^{V/W}$ is a transitive linear group on $V/W$, it has a unique non-abelian simple composition factor deduced by Theorem~\ref{thm:translinear}.
    Hence $G^W\cong G/G_{(W)}$ is solvable, and then $G^{(\infty)}\normeq G_{(W)}$ is simple.

    Suppose that $G^{V/W}\cong G$ is a member in part~(ii) of Theorem~\ref{thm:translinear}.
    Since $G^{(\infty)}$ is simple, one of the following cases holds by Theorem~\ref{thm:translinear}.
    \begin{enumerate}
        \item[(a)] $G^{(\infty)}\cong \A_6$ with $G/G^{(\infty)}\lesssim \bbZ_2$ and $V/W\cong\bbF_2^4$;
        \item[(b)] $G=G^{(\infty)}\cong \A_7$ with  $V/W\cong\bbF_2^4$;
        \item[(c)] $G^{(\infty)}\cong \PSU_3(3)$ with $G/G^{(\infty)}\lesssim \bbZ_2$ and $V/W\cong\bbF_2^6$.
    \end{enumerate}
    We remark that $G^W\cong G/G_{(W)}$ is a quotient of $G/G^{(\infty)}$ as $G^{(\infty)}\normeq G_{(W)}$.
    In the above three cases, $|G/G^{(\infty)}|\leqs 2$ and $p=2$.
    Thus we have that $G^W\cong G/G_{(W)}=1$, and hence $\dim W=1$.
    We can easily construct the matrix form of $\langle Q,G\rangle\leqs\P[W]$ in \Magma\ as follows:
    \[M=\left\{\begin{pmatrix}1&C\\0&B\end{pmatrix}: B\in G^{V/W}\mbox{ and }C\in\mathrm{M}_{1\times d}(\bbF_p)\right\}.\]
    The calculations show that every subgroup of $M$ isomorphic to $G$ has at least $4$ orbits on $V$.
    Therefore, $G\cong G^{V/W}$ is a group satisfying Theorem~{\rm\ref{thm:translinear}\,(i)}.
\end{proof}

Before finalizing the case $G_{(V/W)}=1$, we give the following lemma which helps reduce the possibilities of $G\cong G^{V/W}$.
\begin{lemma}\label{lem:fieldAuto}
    Assume that $G_{(V/W)}=1$ and $G_{(W)}\neq 1$ is almost simple.
    Then $(G^{(\infty)}\cap G_{\overline{v}})^{\overline{v}}\neq 1$ is a $p$-group of exponent $p$ for any $v\in V\setminus W$.
\end{lemma}
\begin{proof}
    Suppose that $h\in G^{(\infty)}\cap G_{\overline{v}}$.
    There exists an element $w_h\in W$ such that $v^h=w_h+v$.
    Note that $G^{(\infty)}\unlhd G_{(W)}$ is non-abelian simple.
    For $h\in G^{(\infty)}\cap G_{\overline{v}}$, we have that $h$ fixes every $w\in W$.
    It follows that $(w+v)^{h^p}=w+p\cdot w_h+v=w+v$, and hence $h^p$ acts trivially on $\overline{v}=v+W$.
    Thus $(G^{(\infty)}\cap G_{\overline{v}})^{\overline{v}}$ is either trivial or a $p$-group of exponent $p$.

    \textbf{Step 1.}
    We claim that $(G_{(W)}\cap G_{\overline{v}})^{\overline{v}}\neq 1$ for any $v\in V\setminus W$.

    By Lemma~\ref{lem:reduceclass}, $G\cong G^{V/W}$ satisfies Theorem~\ref{thm:translinear}\,(i).
    Then $G^{(\infty)}$ is also transitive on non-zero vectors of $V/W$.
    Hence we have that $G=G^{(\infty)}G_{\overline{v}}$.
    Note that
    \[G^W\cong G/G_{(W)}=G^{(\infty)}G_{\overline{v}}/G_{(W)}=G_{(W)}G_{\overline{v}}/G_{(W)}\cong G_{\overline{v}}/(G_{(W)}\cap G_{\overline{v}})\cong (G_{\overline{v}})^W.\]
    Then $G_{\overline{v}}$ is transitive on $W\setminus\{0\}$.
    Suppose that $G_{(W)}\cap G_{\overline{v}}$ acts trivially on $\overline{v}$.
    By Lemma~\ref{lem:crit}, $G_{\overline{v}}$ acts transitively on $\overline{v}$.
    Thus $(G_{\overline{v}})^W$ has a transitive action on $|\overline{v}|=p^d$ points.
    Hence the transitive linear group $(G_{\overline{v}})^W$ has a subgroup of index $p^d$.
    By Lemma~\ref{lem:transindexpd}, we obtain that $G^W\cong (G_{\overline{v}})^W\cong\GL_3(2)$ is non-solvable.
    Recall that $G$ has a unique non-abelian simple composition factor and $G_{(W)}$ is almost simple.
    This leads to a contradiction.
    Thus we have that $(G_{(W)}\cap G_{\overline{v}})^{\overline{v}}\neq 1$.

    \textbf{Step 2.}
    There exist some $v\in V\setminus W$ and $x\in G_{(W)}\setminus G^{(\infty)}$ such that $x\in G_{(\overline{v})}$ and $\langle x\rangle G^{(\infty)}$ is a Sylow $p$-subgroup of $G_{(W)}/G^{(\infty)}$.

    If $G_{(W)}/G^{(\infty)}$ has order coprime with $p$, then we can choose $x=1$.
    Now we can suppose that $G_{(W)}/G^{(\infty)}$ has a non-trivial Sylow $p$-subgroup.
    Recall from Lemma~\ref{lem:reduceclass} that $G_{(W)}$ is an almost simple group and is a transitive linear group on $V/W$ in Theorem~\ref{thm:translinear}\,(i).
    The discussions of the three cases are similar, so we only give a proof for the case where $G^{(\infty)}\cong \SL_{m}(p^f)$ with $mf=n-d$.
    For a complement $U$ of $W$ in $V$, let $L\leqslant \GL(U)$ such that $\langle Q,G_{(W)}\rangle\leqslant \langle Q,L\rangle$ with $L\cong \GammaL_m(p^f)$.
    Let $\phi,\delta\in L$ be such that $L=\langle L^{(\infty)},\phi,\delta\rangle$, where $\delta$ admits a diagonal automorphism and $\phi$ admits a field automorphism.
    We may assume that the action of $\langle\phi\rangle$ on $U\cong\bbF_p^{mf}$ is equivalent of the action of $\Gal(\bbF_{p^f}/\bbF_p)$ on $\bbF_{p^f}^m$.
    Let $L_0\leqslant L$ be such that $\langle Q,G_{(W)}\rangle=\langle Q,L_0\rangle$.
    Since $\gcd(|\delta|, p)=1$, a Sylow subgroup of $L_0/L^{(\infty)}$ is conjugate (by some $y\in L$) to $\langle \phi^k\rangle L^{(\infty)}$ for some $k\geqslant 1$.
    Hence we may assume that $\langle \phi^k\rangle L^{(\infty)}$ is a Sylow $p$-subgroup of $L_0/L^{(\infty)}$ by replacing $G$ with $G^y$.
    Then $\phi^k\in L_0\leqs \langle Q,G_{(W)}\rangle$.
    Note that $\langle Q,G_{(W)}\rangle=Q{:}G_{(W)}$ since $Q$ is normal and $Q\cap G_{(W)}\leqs Q\cap G=O_p(G)=1$.
    It follows that there exists some $z\in Q$ such that $z\phi^k\in G_{(W)}$.
    Note that $\langle Q, z\phi^k\rangle =Q{:}\langle z\phi^k\rangle=Q{:}\langle \phi^k\rangle$ since $\langle z\phi^k\rangle\leqslant G_{(W)}$ and $G_{(W)}\cap Q=1$.
    By Hilbert 90 (see~\cite[Page 97]{morandi1996Field}), we have that
    \[\begin{aligned}
        |\H^1(\langle z\phi^k\rangle,Q)|&=|\H^1(\langle \phi^k\rangle,Q)|=|\H^1(\langle \phi^k\rangle,W\otimes U^*)|\\
        &=|\H^1(\langle \phi^k\rangle,\underbrace{U^*\oplus \cdots \oplus U^*}_{\mbox{$d$-copies}})|=|\H^1(\langle \phi^k\rangle,U)|^d=1.
    \end{aligned}\]
    Hence $Q{:}\langle z\phi^k\rangle$ has a unique conjugacy class of complements of $Q$.
    It follows that there exists $z'\in Q$ such that $(z\phi^k)^{z'}=\phi^k$.
    By replacing $G$ with $G^{z'}$, we may assume that $\phi^k\in G_{(W)}$ and $\langle\phi^k\rangle G^{(\infty)}$ is a Sylow $p$-subgroup of $G_{(W)}/G^{(\infty)}$.
    Set $x=\phi^k\in G_{(W)}$.
    Then $x\in L=\GL(U)$.
    As $\langle x\rangle\leqslant \GL(U)$ is a $p$-group, it has a non-trivial fixed point $u$ in $U$.
    Then $x\in\GL(U)$ fixes all elements in $\overline{u}=u+W$, as desired.

    \textbf{Step 3.}
    Finally, we prove that $(G^{(\infty)}\cap G_{\overline{v}})^{\overline{v}}\neq 1$ for any $v\in V\setminus W$.

    Recall that $G^{(\infty)}$ acts transitively on non-zero vectors of $V/W$.
    Hence we only need to prove that there exists such a vector $v\in V\setminus W$.
    Let $v\in V\setminus W$ and $x\in G_{(W)}\setminus G^{(\infty)}$ such that $x\in G_{(\overline{v})}$ and $\langle x\rangle G^{(\infty)}$ is a Sylow $p$-subgroup of $G_{(W)}/G^{(\infty)}$ as in Step 2.
    Recall from Step 1 that $(G_{(W)}\cap G_{\overline{v}})^{\overline{v}}$ is a $p$-group.
    It follows that
    \[G^{(\infty)}\cap G_{\overline{v}}\unlhd G_{(W)}\cap G_{\overline{v}}\mbox{ and } (G_{(W)}\cap G_{\overline{v}})^{\overline{v}}=(G^{(\infty)}\cap G_{\overline{v}})^{\overline{v}}\langle x\rangle^{\overline{v}}.\]
    Since $x$ acts trivially on $\overline{v}$, we have that $(G_{(W)}\cap G_{\overline{v}})^{\overline{v}}=(G^{(\infty)}\cap G_{\overline{v}})^{\overline{v}}$ is a non-trivial $p$-group of exponent $p$.
    The proof is complete.
\end{proof}

Note that the group $G_1\cong\GL_3(2)$ defined in Construction~\ref{exam:two} acts on $V=\bbF_2^4$ and stabilizes a $1$-dimensional subspace $W$.
Then $G_1$ satisfies Case~(C) in Lemma~\ref{lem:kernel}.
We are ready to prove that $G_1$ is the unique possibility for Case~(C).

\begin{lemma}\label{lem:inftyTrivial}
    Assume that $G_{(V/W)}=1$ and $G_{(W)}\neq 1$ is almost simple.
    Then $V=\bbF_2^4$ and $G$ is conjugate to $G_1$ defined in Construction~$\ref{exam:two}$.
\end{lemma}
\begin{proof}
    By Lemma~\ref{lem:reduceclass}, $G\cong G^{V/W}$ lies in the three infinity families satisfying Theorem~\ref{thm:translinear}\,(i).
    Let $T=G^{(\infty)}\cap G_{\overline{v}}$.
    Lemma~\ref{lem:fieldAuto} shows that $T$ has a normal subgroup $M$ such that $T/M$ is a non-trivial $p$-group of exponent $p$.

    Suppose that $G^{(\infty)}\cong \Sp_{2m}(q)$ with $q=p^f$, $\dim V/W=2mf$, $m\geqs 2$ and $(2m,p^f)\neq (4,2)$.
    Then
    \[T\cong P{:}(K\times L)\cong q^{1+2a}_+{:}((q-1)\times \Sp_{2m-2}(q)),\]
    where $P\cong q^{1+2a}_+$ is the \textit{extraspecial $q$-group} defined in~\cite{li2025finite}.
    Note that $L\cong \Sp_{2m-2}(q)$ acts irreducibly on $P/\Phi(P)\cong p^{2af}$ (see~\cite[Section 4.2]{li2025finite} for instance).
    This yields that $[T,P]=P$, and hence $P\normeq T'$.
    Since $T/M$ is a non-trivial $p$-group, we know $(T/M)'=T'M/M<T/M$ and so $T'M<T$. This implies that $T/T'M$ is a non-trivial elementary $p$-group. Note that $T/T'M$ is isomorphic to a quotient of $K\times L\cong (q-1)\times \Sp_{2m-2}(q)$, it follows that $\Sp_{2m-2}(q)$ has a normal subgroup of index $p^k$ for some $k\geqs 1$.
    Then $(2m,q)=(4,3)$ or $(6,2)$.
    Recall that $G^{(\infty)}$ is a simple group.
    Hence we have that $(2m,q)=(6,2)$ and $G=G^{(\infty)}\cong\Sp_6(2)$.
    This yields that $G=G_{(W)}$ as $G_{(W)}$ is almost simple, and hence $\dim V/W=6$ and $\dim W=1$.
    Then $\langle Q,G\rangle$ has matrix form:
    \[Y=\left\{\begin{pmatrix}1&C\\0&A\end{pmatrix}: A\in \Sp_6(2)\mbox{ and }C\in\M_{1\times 6}(\bbF_2)\right\}\]
    Calculations in \Magma\ show that $Y$ has $2$ subgroups isomorphic to $\Sp_6(2)$ and both of them have $4$ orbits on $V$.

    Suppose that $G^{(\infty)}\cong \G_2(2^f)$ with $f\geqs 2$ and $\dim V/W=6f$.
    Then $T=P{:}L\cong [2^{5f}]{:}\SL_{2}(2^f)$.
    By~\cite[Page 99]{gorenstein1998classification}, we have that $L$ acts irreducibly on $P/\Phi(P)$.
    Reasons similar to above yield that $\SL_2(2^f)$ has a subgroup of index $2^k$ for some $k\geqs 1$.
    This is impossible.

    Now we assume that $G^{(\infty)}\cong \SL_m(p^f)$ with $\dim V/W=mf$.
    Then
    \[T=P{:}L\cong p^{(m-1)f}{:}\SL_{m-1}(p^f)\cong \ASL_{m-1}(p^f).\]
    If $m>2$, then $\SL_{m-1}(p^f)$ has a normal subgroup of index $p^k$ for some $k\geqs 1$ with arguments similar to above.
    This yields that $(m,p^f)=(3,2)$ or $(3,3)$.
    \begin{enumerate}
        \item[(1)] If $(m,p^f)=(3,2)$, then $G=G^{(\infty)}\cong \GL_3(2)$.
        Hence $G^{W}=G/G_{(W)}=1$.
        This yields that $\dim V=4$ and $\dim W=1$.
        From Lemma~\ref{prop:two-2}, we deduce that $G$ is conjugate to $G_1$ defined in Construction~\ref{exam:two}.
        \item[(2)] If $(m,p^f)=(3,3)$, then $G^{(\infty)}\cong \SL_3(3)$.
        Note that $G^W$ is non-trivial as $p\neq 2$.
        It follows that $\dim V=4$, $\dim W=1$ and $G\cong\GL_3(3)$.
        By Proposition~\ref{prop:coho}\,(ii) and calculations on \Magma, we have that
        \[|\H^1(G,Q)|\leqslant |\H^1(G^{(\infty)},W\otimes (V/W)^*)|=|\H^1(\SL_3(3),\bbF_3^3)|=1.\]
        Thus $G$ has at least $4$ orbits on $V$ by Lemma~\ref{lem:conjLevi}, a contradiction.
    \end{enumerate}
    Finally, we suppose that $m=2$.
    Then $G^{(\infty)}\cong\SL_2(p^f)$ and $T\cong p^{f}$.
    Since $G^{(\infty)}$ is a simple group, we have that $p=2$, and hence $G^{(\infty)}\cong \SL_2(2^f)$.
    Let $U$ be the natural $\bbF_{2^f}G^{(\infty)}$-module with basis $u_1,u_2$.
    Then $V/W\cong U$ as $\bbF_{2}G^{(\infty)}$-modules, and let $\varphi: U\rightarrow V/W$ be an $\bbF_2 G^{(\infty)}$-isomorphism.
    For a generator $\lambda$ of $\bbF_{2^f}^\times$, we obtain the following basis of $V/W$:
    \[\varphi(u_1),\varphi(\lambda u_1),...,\varphi(\lambda^{f-1}u_1),\varphi(u_2),\varphi(\lambda u_2),...,\varphi(\lambda^{f-1}u_2).\]
    Let $v_{d+1},...,v_n$ be preimages of the above vectors in $V$, respectively.
    Then we can obtain a basis $v_1,...,v_n$ of $V$ by choosing a basis $v_1,...,v_d$ of $W$.
    We may assume that $v=v_{d+1}$.
    Then, with respect to this basis, each element $x\in L=G^{(\infty)}_{\overline{v}}$ has matrix form:
    \[\M(x)=\begin{pmatrix}I & \alpha_x & \beta_x \\ & I & P_x \\ & & I\end{pmatrix}\mbox{ where }\alpha_x,\beta_x\in \mathbb{F}_2^{d\times f}\mbox{ and }P_x\in \mathrm{M}_{f\times f}(\mathbb{F}_2).\]
    Note that there exists some $y\in L$ such that $\varphi(\lambda^k u_2)^y=\varphi(\lambda^k u_1)+\varphi(\lambda^k u_2)$ for $k=0,...,f-1$.
    Hence we have that $\M(y)=\begin{pmatrix}I & \alpha_y & \beta_y \\ & I & I \\ & & I\end{pmatrix}$.
    Note that $|y|=2$ as $L$ is an elementary abelian $2$-group.
    By calculations, we have that $\M(y)^2=\begin{pmatrix}I &  & \alpha_y \\ & I & \\ & & I\end{pmatrix}$ and it follows that $\alpha_y=0$.
    Let $H=\bfN_{G^{(\infty)}}(L)\cong\AGL_1(2^f)$.
    Then, for any $1\neq x\in L$, there exists some $h\in H$ such that $x^h=y$ and
    \[\M(h)=\begin{pmatrix}I & \alpha_h & \beta_h \\ & \Lambda_h & \\ & & \Lambda_h^{-1}\end{pmatrix}\mbox{, where $\Lambda_h\in\GL_f(2)$.}\]
    By calculation, we have that
    \[\M(y)=\M(x^h)=\begin{pmatrix}I & \alpha_x\Lambda_h & \beta_x \Lambda_h^{-1} \\ & I & \Lambda_h^{-1}P_x\Lambda_h^{-1}\\ & & I\end{pmatrix}.\]
    Then $\Lambda_h^{-1}P_x\Lambda_h^{-1}=I$ and $\alpha_x\Lambda_h=\alpha_y=0$.
    This yields that $\alpha_x=0$ as $\Lambda_h$ is invertible.
    Thus we observe that $x$ acts trivially on $\overline{v}=v+W$ for any $x\in L$, a contradiction.

    Therefore, $G$ is conjugate to $G_1$ defined in Construction~\ref{exam:two}.
\end{proof}

We conclude the proof of Theorem~\ref{thm:nonplocal} as below.

\begin{proof}[Proof of Theorem~$\ref{thm:nonplocal}$.]
    Recall that both $G_1$ and $G_2$ defined in Construction~\ref{exam:two} are reducible and have $3$ orbits on $V=\bbF_2^4$ by Lemma~\ref{prop:two}.
    We only need to prove the sufficiency of the theorem.

    Suppose that $G$ is reducible with $O_p(G)=1$ and has $3$ orbits on $V$.
    Assume that $G$ stabilizes a non-trivial subspace $W<V$.
    Then exactly one of $G_{(W)}$ and $G_{(V/W)}$ is trivial by Lemma~\ref{lem:bothfaith}.
    If $G_{(W)}=1$, then $V=\bbF_2^4$ and $G$ is conjugate to $G_2$ in Construction~\ref{exam:two} by Lemmas~\ref{lem:kernel} and~\ref{lem:faithfulW}; and if $G_{(V/W)}=1$, then $V=\bbF_2^4$ and $G$ is conjugate to $G_1$ in Construction~\ref{exam:two} by Lemmas~\ref{lem:kernel} and~\ref{lem:inftyTrivial}.
\end{proof}

\end{document}